\documentclass[12pt,a4paper,oneside]{scrartcl}
\usepackage[utf8]{inputenc}
\usepackage{amsmath}
\usepackage{amsfonts}
\usepackage{amssymb}
\usepackage{graphicx}
\usepackage[english]{babel}
\usepackage{color}
\usepackage{tikz}
\usepackage[T1]{fontenc}
\usepackage{nicefrac}
\usepackage{hyperref}
\usepackage{cite}

\newcommand{\re}{\mathbb{R}}
\newcommand{\na}{\mathbb{N}}

\newcommand{\eps}{\varepsilon}

\newcommand{\diver}{\operatorname{div}}

\newcommand{\HS}{\text{HS}}
\newcommand{\sobolevp}{W^{1,p}(\re^d)}
\newcommand{\sobolevq}{W^{-1,p'}(\re^d)}
\newcommand{\dx}{\, dx}

\newcommand{\erwb}{\mathbb{E}\left[}
\newcommand{\erwe}{\right]}
\newcommand{\otr}{\Omega\times(0,T)\times\re^d}
\newcommand{\halbe}{\frac{1}{2}}
\newcommand{\ueps}{u_\eps}

\newcommand{\nablas}[1]{|\nabla{#1}|^{p-2}\nabla{#1}}

\newcommand{\ranglev}{\rangle_{\mathcal{V}',\mathcal{V}}}

\newcommand{\chitk}{\chi_{[t_k,t_{k+1})}}
\newcommand{\otwr}[1]{\Omega\times(0,T);L^{#1}(\re^d)}

\newtheorem{definition}{Definition}[section]

\newtheorem{lemma}[definition]{Lemma}
\newtheorem{theorem}[definition]{Theorem}
\newtheorem{proposition}[definition]{Proposition}

\newenvironment{proof}{\noindent{\textit{Proof.}}}{\hfill$\square$}

\begin{document}
\title{The stochastic $p$-Laplace equation on $\re^d$}

\author{Kerstin Schmitz, \footnote{Faculty of Mathematics, University of Duisburg-Essen, Thea-Leymann-Str.9, 45127 Essen, Germany
\href{mailto:kerstin.schmitz@uni-due.de}{\nolinkurl{ kerstin.schmitz@uni-due.de}}} \and Aleksandra Zimmermann \footnote{Faculty of Mathematics, University of Duisburg-Essen, Thea-Leymann-Str.9, 45127 Essen, Germany  \href{mailto:aleksandra.zimmermann@uni-due.de}{\nolinkurl{ aleksandra.zimmermann@uni-due.de}}}}
\date{}
\maketitle

\begin{abstract}
We show well-posedness of the $p$-Laplace evolution equation on $\mathbb{R}^d$ with square integrable random initial data for arbitrary $1<p<\infty$ and arbitrary space dimension $d\in\na$. The noise term on the right-hand side of the equation may be additive or multiplicative. Due to a lack of coercivity of the $p$-Laplace operator in the whole space, the possibility to apply well-known existence and uniqueness theorems in the classical functional setting is limited to certain values of $1<p<\infty$ and also depends on the space dimension $d$. We propose a framework of functional spaces which is independent of Sobolev space embeddings and space dimension. For additive noise, we show existence using a time discretization. Then, a fixed-point argument yields the result for multiplicative noise.
\end{abstract}
\begin{flushleft}
\textbf{Keywords:} Stochastic $p$-Laplace equation, whole space, monotonicity method, time discretization
\end{flushleft}
\begin{flushleft}
\textbf{2020 Mathematics subject classification:} 60H15, 35R60, 35K55
\end{flushleft}

\section{Introduction}
\label{intro}
In the theory of filtration of an elastic fluid in a porous medium (see, \cite{Barenblatt} Section 3.2.1) the linear Darcy law expresses that the velocity of filtration is proportional to the pressure gradient. In the presence of a heterogeneous medium or turbulence, a generalized, $p$-power type version of the Darcy law has been proposed in \cite{Diaz} (see also the references therein). Substituting into the continuity equation, normalizing constants and discarding lower-order terms we find  
\begin{equation}\label{P}
\partial_t u-\operatorname{div}(|\nabla u|^{p-2}\nabla u)=f
\end{equation}
where $u=u(t,x)$ is the unknown function that may be interpreted as the volumetric moisture content (see \cite{EV}), $1<p<\infty$ varies with the properties of the flow and $f=f(t,x)$ models an external force. For fixed $p\in (1,\infty)$, the second order diffusion operator $\Delta_p(u):=\operatorname{div}(|\nabla u|^{p-2}\nabla u)$ on the left-hand side of \eqref{P} is called $p$-Laplace operator. On a spatial domain $D\subseteq\mathbb{R}^d$, $-\Delta_p(u)$ is a monotone, variational operator on the Sobolev space 
\[W^{1,p}(D)=\left\{u\in L^p(D) : \ \frac{\partial u}{\partial{x_i}}\in L^p(D), \ i=1,\dots,d\right\},\]
for all $1<p<\infty$, and this operator is singular for $p<2$ because $|\nabla u|^{p-2}=+\infty$ for $\nabla u=0$ and degenerate for $p>2$. In the case of a bounded domain, the initial-boundary value problem for \eqref{P} with square integrable initial condition, Dirichlet or Neumann type boundary values and square integrable right-hand side is a classical problem of monotone nonlinear evolution equations and has already been studied in \cite{Lad} and \cite{Lions}. In the following decades, monotonicity methods have been applied to more general operators of Leray-Lions type, nonlinear or set valued boundary values and merely integrable initial data and right-hand sides. In some cases, well-posedness could be established for generalized notions of solution such as entropy solutions, see, e.g., \cite{AMST99} and renormalized solutions (see, e.g., \cite{B93, BMR01, AW03}).\\
There are more and more stochastic models in natural, human and also social sciences. One way to take the complexity of the phenomena, the uncertainties of the model and the effects of multiscale interactions into account is to add random influences (see, e.g., \cite{Geigeretal} and \cite{Tartakovsky}). Randomness can be introduced by a stochastic forcing of the driving partial differential equation (PDE), e.g., by adding a stochastic integral on the right-hand side. Consequently, the PDE becomes a stochastic partial differential equation (SPDE) and the solution is then a stochastic process.\\ 
In this contribution, our aim is to study the stochastic $p$-Laplace evolution equation on $\mathbb{R}^d$, i.e.,
\begin{equation}\label{stP}
du-\operatorname{div}(|\nabla u|^{p-2}\nabla u)\, dt=\Phi\,dW_t \quad \text{in} \ \Omega\times (0,T)\times\mathbb{R}^d
\end{equation}
for fixed $T>0$, a fixed stochastic basis $(\Omega,\mathcal{F},P,(\mathcal{F}_t)_{t\geq 0}, (W_t)_{t\geq 0})$, arbitrary space dimension $d\in\mathbb{N}$ and arbitrary $1<p<\infty$. The right-hand side of \eqref{stP} is understood in the sense of a stochastic It\^{o} integral with respect to the cylindrical Wiener process $(W_t)_{t\geq 0}$ with values in $L^2(\mathbb{R}^d)$. We will consider both additive noise, i.e., $\Phi$ is progressively measurable with values in the space of Hilbert-Schmidt operators from $L^2(\mathbb{R}^d)$ to $L^2(\mathbb{R}^d)$ and multiplicative noise, i.e., $\Phi=B(u)$ and the precise assumptions on $B$ are given in Subsection \ref{1.1}.\\
In \cite{BT72} and \cite{PardouxThese}, the method of monotonicity which was developed in \cite{Lions} for deterministic equations has been extended to stochastic evolution equations. The results have been generalized in \cite{KryRoz81} and \cite{LiuRock}:
For an appropriately chosen separable Banach space $V$ with dual space $V'$ and an intermediate Hilbert space $H$ such that $V\hookrightarrow H\equiv H'\hookrightarrow V'$ with continuous and dense embeddings one considers the stochastic evolution equation driven by a nonlinear operator $A:V\rightarrow V'$. The key properties for well-posedness (i.e., existence and uniqueness of solutions) are the monotonicity, the coercitivity and the growth condition of the operator $A$. For $A(u)=-\Delta_p(u)=-\operatorname{div}(|\nabla u|^{p-2}\nabla u)$ with $1<p<\infty$, on a bounded domain $D\subset\mathbb{R}^d$ with homogeneous Dirichlet boundary data a natural choice is $V=W^{1,p}_0(D)\cap L^2(D)$, where $W^{1,p}_0(D)$ denotes the closure of the test functions with respect to the norm in $W^{1,p}(D)$. From the Sobolev embedding theorem it follows that $V=W^{1,p}_0(D)$ for $p>\frac{2d}{d+2}$. In this setting, the monotonicity and the growth condition are a consequence of the fundamental algebraic inequality
\[(|a|^{p-2}a-|b|^{p-2}b)\cdot(a-b)\geq 0\] 
that holds for all $a,b\in\mathbb{R}^d$ and $1<p<\infty$. Coercivity follows from the Poincar\'e inequality.\\
However, on $\mathbb{R}^d$, the Poincar\'e inequality and certain Sobolev space embeddings do not hold true and the question of appropriate functional spaces such that the coercitivity and the growth condition hold true for all values of $1<p<\infty$ and in arbitrary space dimension $d\in\mathbb{N}$ becomes delicate. In this contribution, we propose a functional setting for the $p$-Laplace operator on $\mathbb{R}^d$ and we show that, in our setting, the coercitivity and growth conditions proposed in \cite{LiuRock, Barbu, Ren} do not hold both at the same time for $1<p<\infty$ in a certain range. Then we prove well-posedness, i.e., existence and uniqueness of solutions to \eqref{stP} in the proposed setting with both an additive stochastic perturbation and a multiplicative stochastic perturbation.\\
For the porous medium equation on $\mathbb{R}^d$, existence and uniqueness have been studied in \cite{BaRoRu15}. Therein, the porous medium operator $\Delta\Psi(u)$ has been considered for $\Psi:\mathbb{R}\rightarrow\mathbb{R}$ monotonically nondecreasing and Lipschitz continuous or eventually maximal monotone with appropriate polynomial growth. The linear multiplicative noise term in \cite{BaRoRu15} is given by a stochastic It\^{o} integral with respect to a $Q$-Wiener process on $H^{-1}(\mathbb{R}^d)$.\\
 
In the special case $p=2$, equation \eqref{stP} becomes the stochastic heat equation
\begin{align*}\tag{SHE}
du-\Delta u\, dt=\Phi\,dW_t
\end{align*}
which has been studied intensively using the semigroup representation of solutions, see, e.g., \cite{Iwata87, Reimers89, Shiga94, Myt06, Myt11} and the list is far from being complete.\\
Over the past years, the theories of rough paths and regularity structures have developed quickly and allow to give consistent interpretations for a number of ill-posed equations, including singular cases of (SHE) with space-time white noise. Already in space dimension $d=1$, mollified versions of the multiplicative stochastic heat equation do not converge (see, e.g., \cite{HP15}). In \cite{HL18}, (SHE) on the whole space has been addressed in the framework of regularity structures and in \cite{DGT} and \cite{FNSt} in the framework of rough paths theory (see, e.g., \cite{GIP} for more general information on the topic).\\
Passing from (SHE) to \eqref{stP} for arbitrary $1<p<\infty$, the regularity of solutions does not improve, while nonlinear expressions of the gradient enter into the leading second order operator. Hence, a semigroup representation of the mild solution is not available and many interesting questions concerning singular versions of \eqref{stP} remain open.\\

In our setting, the multiplicative noise term is Lipschitz continuous and the Wiener process may be interpreted to be white in time and coloured in space. However, our aim is to establish a functional setting which allows a further investigation of the stochastic $p$-Laplace equation on the whole space with more general (additive and multiplicative) noise.

\subsection{Results and outline}\label{1.1}

We study the stochastic $p$-Laplace equation on $\re^d$ 
\begin{align}\tag{SPDE}\label{stochplaplaceequation}
\begin{aligned}
d u-\diver(|\nabla u|^{p-2}\nabla u)\,dt&+f(u)\, dt=\Phi\,dW_t \quad &\text{in } \Omega\times(0,T)\times\re^d\\
u(0,\cdot)&=u_0 \quad&\text{in } \Omega\times\re^d
\end{aligned}
\end{align}
for $T>0$, with arbitrary $1<p<\infty$, in any space dimension $d\in\na$ on a probability space $(\Omega,\mathcal{F},P)$ endowed with a right-continuous, complete filtration $(\mathcal{F}_t)_{t\geq 0}$ and a cylindrical Wiener process $(W_t)_{t\geq 0}$ with respect to $(\mathcal{F}_t)_{t\geq 0}$ taking values in $L^2(\re^d)$. The right-hand side of \eqref{stochplaplaceequation} is understood in the sense of It\^{o}. The initial condition $u_0$ is in $L^2(\Omega\times\re^d)$ and $\mathcal{F}_0$-measurable. 
First, we will consider additive stochastic noise for a general, progressively measurable $\Phi\in L^2(\Omega\times(0,T);\HS(L^2(\re^d)))$,
where $\HS(L^2(\re^d))$ denotes the space of Hilbert-Schmidt operators from $L^2(\re^d)$ to $L^2(\re^d)$. Then, we will consider multiplicative noise, i.e., $\Phi=B(u)$ where $B:\Omega\times (0,T)\times L^2(\mathbb{R}^d)\rightarrow \HS(L^2(\re^d))$ satisfies the following conditions:
\begin{itemize}
\item[$(B1)$] $B$ is a Carath\'{e}odory function in the sense that 
\[(0,T)\times L^2(\mathbb{R}^d)\ni (t,v)\mapsto B(\omega,t,v)\] is continuous for almost every $\omega\in\Omega$ and 
\[\Omega\times (0,T)\ni(\omega,t)\mapsto B(\omega,t,v)\] 
is progressively measurable for every $v\in L^2(\mathbb{R}^d)$.
\item[$(B2)$] $B$ is Lipschitz continuous, i.e., there exists $L>0$ such that
\[\Vert B(\omega,t,v)-B(\omega,t,w)\Vert_{\HS(L^2(\mathbb{R}^d))}\leq L \Vert v-w\Vert_{L^2(\mathbb{R}^d)}\]
for all $v,w\in L^2(\mathbb{R}^d)$ and all $t\in (0,T)$ and almost every $\omega\in\Omega$. 
\item[$(B3)$] $B$ has sublinear growth, i.e., there exists $C\geq 0$ such that
\[\Vert B(\omega, t,v)\Vert_{\HS(L^2(\mathbb{R}^d))}\leq C(1+\Vert v\Vert_{L^2(\mathbb{R}^d)})\]
for all $v\in L^2(\mathbb{R}^d)$ and all $t\in (0,T)$ and almost every $\omega\in\Omega$.
\end{itemize}
As usual, we will not write the variable $\omega$ when it is not relevant for the argumentation.
Thanks to $(B1)$, $B$ is $\mathcal{F}\times \mathcal{B}(0,T)\times\mathcal{B}(L^2(\mathbb{R}^d))/\mathcal{B}(\HS(L^2(\mathbb{R}^d)))$-measurable. Moreover, for any progressively measurable function $u$ in $L^2(\Omega\times (0,T);L^2(\mathbb{R}^d))$, 
\[\Omega\times (0,T)\ni(\omega,t)\mapsto B(\omega,t,u(\omega,t,\cdot))\in \HS(L^2(\mathbb{R}^d))\] 
is progressively measurable and, thanks to $(B3)$, in $L^2(\Omega\times(0,T);\operatorname{HS}(L^2(\mathbb{R}^d)))$.\\

The challenge of the stochastic $p$-Laplace equation in the whole space $\re^d$ lies in the absence of Poincar\'{e}'s inequality and certain Sobolev space embeddings: We recall that, for space dimension $d=2$ and $1\leq p<2$ as well as for space dimension $d>2$ and $1\leq p<d$ such that $p\in [\frac{2d}{d+2},2]$ we have the continuous embedding  $W^{1,p}(\mathbb{R}^d)\hookrightarrow L^2(\mathbb{R}^d)$ (see, e.g. \cite[Theorem 9.9, p.278 and Corollary 9.10, p.281]{brezis}). In the case $p=d$, the embedding $W^{1,p}(\mathbb{R}^d)\hookrightarrow L^2(\mathbb{R}^d)$ holds true for $p=d=1$ or $p=d=2$ and is in general not true if $p>d$ and $p>2$ (see, e.g., \cite[Corollary 9.11, p.281 and Theorem 9.12, p.282]{brezis}). To overcome this difficulty, for any $p,p'\in (1,\infty)$ such that $\frac{1}{p}+\frac{1}{p'}=1$ one may consider the space
\[\mathcal{V}:=\{v\in L^2(\mathbb{R}^d) \, : \, \nabla v\in L^p(\mathbb{R}^d)^d\}\]
endowed with the norm 
\[\Vert v\Vert_{\mathcal{V}}:=\Vert v\Vert_2+\Vert\nabla v\Vert_p, \ v\in \mathcal{V}.\]
It is not very hard to see that $\mathcal{V}$ is a reflexive and separable Banach space with dual space
\[\mathcal{V}'=\{f-\operatorname{div}\,F \, : \, f\in L^2(\mathbb{R}^d), \ F\in L^{p'}(\mathbb{R}^d)^d\}\]
and the $p$-Laplace operator
\begin{align*}
\Delta_p:\mathcal{V}&\longrightarrow \mathcal{V}'\\
v&\longmapsto\diver(|\nabla v|^{p-2}\nabla v)
\end{align*}
is well-defined.
\begin{definition}\label{solution}
A solution to \eqref{stochplaplaceequation} for arbitrary $1<p<\infty$ is a square-integrable, $(\mathcal{F}_t)_{t\geq 0}$-adapted stochastic process $u:\Omega\times [0,T]\rightarrow L^2(\mathbb{R}^d)$ with a.s. continuous paths such that
$u\in L^q(\Omega\times (0,T);\mathcal{V})$ for $q=\min\{p,2\}$, $u(0)=u_0$ and 
\[u(t)-u_0-\int_0^t \Delta_p(u(s))\,ds=\int_0^t \Phi(s)\,dW_s\]
for all $t\in [0,T]$, a.s. in $\Omega$.  
\end{definition}
Our main results are
\begin{theorem}\label{existence}
For any $\mathcal{F}_0$-measurable initial condition $u_0$ in $L^2(\Omega\times\re^d)$ and arbitrary $1<p<\infty$ there exists a solution to \eqref{stochplaplaceequation} in the sense of Definiton \ref{solution}.
\end{theorem}
\begin{theorem}\label{uniquesolution}
For any $\mathcal{F}_0$-measurable initial condition $u_0$ in $L^2(\Omega\times\re^d)$ for arbitrary $1<p<\infty$ the solution to \eqref{stochplaplaceequation} in the sense of Definiton \ref{solution} is unique.
\end{theorem}
\subsection{Novelty of the approach}
In the following we will show that the classical well-posedness theory for monotone operators, see, e.g., \cite[Theorem 4.2.4, p.91; Theorem 5.1.3, p.125; Theorems 5.2.2, 5.2.4 and 5.2.6, p.146-148]{LiuRock} and \cite[Theorem 4.20, p.183]{Barbu} and also the more general framework \cite{Ren}, can not be applied to \eqref{stochplaplaceequation} in the framework of the Gelfand triple $\mathcal{V}\hookrightarrow L^2(\mathbb{R}^d)\hookrightarrow \mathcal{V}'$. According to the notation of \cite{Ren}, for $q:=\min(p,2)$, we define $K:=L^q(\Omega\times (0,T);\mathcal{V})$,
\[R:\mathcal{V}\rightarrow[0,\infty),\ v\mapsto R(v):=\left(\|v\|_{L^2(\re^d)}+\|\nabla v\|_{L^p(\re^d)}\right)^q.\]
With straightforward computation one can show that $R$ satisfies condition (K), (i)-(iv) in \cite{Ren}.\\
Moreover it is not hard to see that $A:\mathcal{V}\rightarrow\mathcal{V}'$ defined by $A(v):=\Delta_p(v)$ satisfies the hemicontinuity condition (H1) and the weak monotonicity condition (H2) of \cite{Ren} for all $1<p<\infty$. Now, let us assume that that the coercivity condition (H3)
\begin{align*}
\langle A(v),v\rangle_{\mathcal{V}',\mathcal{V}}+\|\Phi\|_\HS^2\leq c_1\|v\|^2_{L^2(\re^d)}-c_2R(v)\quad\text{for all }v\in \mathcal{V}
\end{align*}
for constants $c_1,c_2>0$ and the boundedness condition (H4) 
\begin{align*}
|\langle A(v),u\rangle_{\mathcal{V}',\mathcal{V}}|\leq c_3(R(v)+R(u))\quad \text{for all }u,v\in\mathcal{V}
\end{align*}
for a constant $c_3>0$ hold true. For any $v\in\mathcal{V}$, plugging $u=v\in\mathcal{V}$ in (H4) we obtain
\begin{align}\label{eqscalingargument}
\|\nabla v\|_{L^p(\re^d)}^p\leq 2c_3\left(\|v\|_{L^2(\re^d)}+\|\nabla v\|_{L^p(\re^d)}\right)^q.
\end{align} 
Let $v\in C_c^\infty(\re^d)$ and $v_\lambda(x):=v(\lambda x)$ for arbitrary $\lambda>0,x\in\re^d$. For $p>\max(2,d)$ (hence $q=2$), plugging $v_{\lambda}$ into  \eqref{eqscalingargument} we get
\begin{align*}
\|\nabla v\|_{L^p(\re^d)}^p&\leq 4 c_3 \lambda^{d-p}\left(\lambda^{-d}\|v\|_{L^2(\re^d)}^2+\lambda^{\frac{2(p-d)}{p}}\|\nabla v\|_{L^p(\re^d)}^2\right)\\
&\leq 4c_3 \left(\lambda ^{-p}+\lambda^{\frac{(d-p)(p-2)}{p}} \right)(\|v\|_{L^2(\re^d)}^2+\|\nabla v\|_{L^p(\re^d)}^2).
\end{align*}
Letting $\lambda\rightarrow\infty$ yields a contradiction since $(d-p)(p-2)<0$ for $p>\max(2,d)$ and therefore the right-hand side in the above equation tends to zero. In the case that $1<p<\frac{2d}{d+2}$ and $d>2$, plugging $v_\lambda$ in (H3) and taking into account that $q=p$ we obtain
\begin{align*}
-c_1\lambda^{-p}\|v\|_{L^2(\re^d)}^2+c_2\left(\lambda^{\frac{d-p}{p}-\frac{d}{2}}\|v\|_{L^2(\re^d)}+\|\nabla v\|_{L^p(\re^d)}\right)^p\leq \|\nabla v\|_{L^p(\re^d)}^p.
\end{align*}
For $\lambda\rightarrow\infty$, $\lambda^{-p}$ tends to zero, but $\lambda^{\frac{d-p}{p}-\frac{d}{2}}$ tends to infinity since $\frac{d-p}{p}-\frac{d}{2}>0$ if and only if $p<\frac{2d}{d+2}$. Therefore, we get a contradiction.\\
In the special case of the continuous embedding  $W^{1,p}(\mathbb{R}^d)\hookrightarrow L^2(\mathbb{R}^d)$ (which is not true for $p>\max(d,2)$ and $1<p<\frac{2d}{d+2}$) it is possible to apply the monotonicity method of \cite{LiuRock} for a regularized problem with an additional term and to show existence of solutions using the vanishing viscosity method. For the sake of completeness, this approach will be presented in the Appendix.\\
In the sequel, we will proof Theorem \ref{existence} and Theorem \ref{uniquesolution} for an additive noise and arbitrary $1<p<\infty$. We will solve time-discretised approximations and pass to the limit to show existence of solutions in the framework of the Gelfand triple $\mathcal{V}\hookrightarrow L^2(\mathbb{R}^d)\hookrightarrow \mathcal{V}'$. With the well-posedness for the additive case at hand, existence and uniqueness for the multiplicative case will be proved with fixed point arguments.

\section{Additive noise}
In this section we consider \eqref{stochplaplaceequation} with $1<p<\infty$ and additive noise: 
\begin{align}\tag{SPDEa}\label{spdeadditiv}
\begin{aligned}
du-\diver(\nablas{u})\,dt&=\Phi\,dW_t &\text{in }\Omega\times(0,T)\times\re^d\\
u(0,\cdot)&=u_0 &\text{in }\Omega\times\re^d,
\end{aligned}
\end{align}
where
\[\Phi:\Omega\times(0,T)\rightarrow \HS(L^2(\re^d))\]
is progressively measurable, and there exists a constant $\mathrm{C}_{\Phi}$ such that
\begin{align*}
\int_0^T\erwb\|\Phi\|_\HS^2\erwe\,dt\leq\mathrm{C}_{\Phi}.
\end{align*}
In particular, $\Phi\in L^2(\Omega;L^2((0,T);\HS(L^2(\re^d))))$. 

\subsection{Proof of Theorem \ref{uniquesolution}}
\begin{lemma}\label{uniqueness}
The solution to \eqref{spdeadditiv} is unique.
\end{lemma}

\begin{proof}
Let $u_1$ and $u_2$ be solutions to \eqref{spdeadditiv} with initial value $u_0$, then we have
\begin{align*}
d(u_1-u_2)-\diver(|\nabla u_1|^{p-2}\nabla u_1-|\nabla u_2|^{p-2}\nabla u_2)\,dt=0.
\end{align*}
Using $u_1-u_2$ as a test function in the above equation we get 
\begin{align*}
&\int_{\re^d}|u_1(t)-u_2(t)|^2\,dx\\
&+\int_0^t\int_{\re^d}(|\nabla u_1|^{p-2}\nabla u_1-|\nabla u_2|^{p-2}\nabla u_2)\nabla (u_1-u_2)\,dx\,ds=0
\end{align*}
for any for $t\in(0,T)$. This in combination with the fundamental inequality
\begin{align}\label{technicalineq}
(|\eta|^{p-2}\eta-|\xi|^{p-2}\xi)(\eta-\xi)\geq 0
\end{align}
which holds for all $\eta,\xi\in\re^d$, provides
$u_1(t)=u_2(t)$ for all $t\in[0,T]$, a.e. in $\re^d$, a.s. in $\Omega$.
\end{proof}

\subsection{Proof of Theorem \ref{existence}}

We recall the Gelfand triple
\begin{align*}
\mathcal{V}\hookrightarrow L^2(\re^d)\hookrightarrow \mathcal{V}',
\end{align*}
where $\mathcal{V}=\left\{v\in L^2(\re^d):\nabla v\in L^p(\re^d)^d\right\}$
with the norm $\|v\|_{\mathcal{V}}:=\|v\|_2+\|\nabla v\|_p$ and dual space
$\mathcal{V}'=\left\{f-\diver F:f\in L^2(\re^d),F\in L^{p'}(\re^d)^d\right\}$
as defined in the previous section. 

\subsubsection{The time discretization}
For $N\in\na$, we divide the interval $[0,T]$ in $0=t_0<t_1<\dots<t_N=T$ equidistantly with $\tau_N:=t_{k+1}-t_k=\frac{T}{N}$ for $k\in\{0,\dots,N-1\}$. Moreover, we define $\Phi^0:=0$, $\Phi^k:=\frac{1}{\tau_N}\int_{t_{k-1}}^{t_k}\Phi(s)\,ds$ for $k\in\{1,\dots,N\}$ and $\triangle_{k+1}W:=W_{t_{k+1}}-W_{t_k}$, $k\in\{0,\dots,N-1\}$. For almost every $\omega\in\Omega$, we consider the stationary problem
\begin{align}\tag{TDE}\label{timediscretqu}
u^{k+1}-u^k-\tau_N\diver(\nablas{u^{k+1}})&=\Phi^k\triangle_{k+1}W.
\end{align}
We remark that since it can be shown that all the members of \eqref{timediscretqu} are in $L^2(\mathbb{R}^d)$, \eqref{timediscretqu} can be interpreted as a random equation in $L^2(\mathbb{R}^d)$. 

\begin{proposition} 
For any $N\in\mathbb{N}$, $k\in\{0,\dots,N-1\}$ and any $\mathcal{F}_{t_k}$-measurable random variable $u^k$ with values in $L^2(\mathbb{R}^d)$ there exists a $\mathcal{F}_{t_{k+1}}$-measurable random variable $u^{k+1}\in \mathcal{V}$ satisfying \eqref{timediscretqu}
in $L^2(\mathbb{R}^d)$, a.s. in $\Omega$.
\end{proposition}

\begin{proof}
For $N\in\mathbb{N}$, we define the operator
\begin{align*}
A_{\tau_N}:\mathcal{V}&\longrightarrow \mathcal{V}'\\
v&\longmapsto v-\tau_N\Delta_p(v).
\end{align*}
It follows immediately that $A_{\tau_N}$ is hemicontinuous. The strict monotonicity of $A_{\tau_N}$ follows from \eqref{technicalineq}. The boundedness of $A_{\tau_N}$ follows directly from Hölder's inequality. 
To show the coercivity of $A_{\tau_N}$, we choose an arbitrary $v\in \mathcal{V}$ and find
\begin{align}\label{200907_01}
\begin{aligned}
\frac{\langle A_{\tau_N}(v),v\ranglev}{\Vert v\Vert_{\mathcal{V}}} &=\frac{\Vert v\Vert_2^2+\tau_N\Vert\nabla v\Vert_p^p}{\Vert v\Vert_2+\Vert\nabla v\Vert_p}\\
&=\Vert v\Vert_2\left(\frac{\Vert v\Vert_2}{\Vert v\Vert_2+\Vert\nabla v\Vert_p}\right)+\tau_N\Vert\nabla v\Vert_p^{p-1}\left(\frac{\Vert \nabla v\Vert_p}{\Vert v\Vert_2+\Vert\nabla v\Vert_p}\right).
\end{aligned}
\end{align}
For $\|v\|_{\mathcal{V}}\rightarrow\infty$, the first or the second term on the right-hand side of \eqref{200907_01} goes to $\infty$ and consequently
\begin{align*}
\lim_{\|v\|_{\mathcal{V}}\rightarrow\infty}\frac{\langle A_{\tau_N}(v),v\ranglev}{\|v\|_{\mathcal{V}}}=\infty.
\end{align*}
Now we can apply the theorem of Minty-Browder (see \cite[Satz 1.5, p.63]{Ruzicka}), that provides, that for each $f\in \mathcal{V}'$ there exists a unique element $v\in \mathcal{V}$ such that $A_{\tau_N}(v)=f$. If we define for a.e. $\omega\in\Omega$
\begin{align*}
f^k:=\Phi^k\triangle_{k+1}W+u^k\in \mathcal{V}',
\end{align*}
there exists $u^{k+1}:\Omega\rightarrow \mathcal{V}$ such that $A_{\tau_N}(u^{k+1})=f^k$, a.s. in $\Omega$, i.e. \eqref{timediscretqu} holds a.e. in $\Omega$.
If we can show that the inverse operator $A_{\tau_N}^{-1}:\mathcal{V}'\rightarrow \mathcal{V}$ is demi-continuous, it follows immediately from the Pettis Theorem (see, \cite{Yosida}, V4, p.131) and from the $\mathcal{F}_{t_k}$-measurability of $u^k$ that $u^{k+1}$ is $\mathcal{F}_{t_{k+1}}$-measurable, since 
\[u^{k+1}=A_{\tau_N}^{-1}\left(\Phi^k\triangle_{k+1}W+u^k\right).\]
This will be shown in the following lemma.
\end{proof}

\begin{lemma}
$A_{\tau_N}^{-1}:\mathcal{V}'\rightarrow \mathcal{V}$ is demi-continuous.
\end{lemma}

\begin{proof}
We choose a sequence $(f_n)_{n\in\na}\subseteq \mathcal{V}'$ such that $f_n\rightarrow f$ as $n\rightarrow\infty$ in $\mathcal{V}'$. For $n\in\na$, we define $u_n:=A_{\tau_N}^{-1}(f_n)$. First we remark, that $(u_n)_{n\in\na}$ is bounded in $\mathcal{V}$: 
Assume, that $(u_n)_{n\in\na}$ is not bounded in $\mathcal{V}$, then, passing to a not relabeled subsequence if necessary, from the coercivity of $A_{\tau_N}$ it follows that
\begin{align*}
\lim_{n\rightarrow\infty}\frac{\langle A_{\tau_N}(u_n),u_n\ranglev}{\|u_n\|_{\mathcal{V}}}=\infty.
\end{align*}
On the other hand
\begin{align*}
\frac{\langle A_{\tau_N}(u_n),u_n\ranglev}{\|u_n\|_{\mathcal{V}}}=\frac{\langle f_n,u_n\ranglev}{\|u_n\|_{\mathcal{V}}}\leq\|f_n\|_{\mathcal{V}'}\rightarrow \|f\|_{\mathcal{V}'}\text{ as }n\rightarrow\infty.
\end{align*}
This is a contradiction and therefore $(u_n)_{n\in\na}$ is bounded in $\mathcal{V}$. Since $\mathcal{V}$ is reflexive, there exists a subsequence, still denoted by $(u_n)_{n\in\na}$, and $u\in \mathcal{V}$ such that
\begin{align*}
u_n\rightharpoonup u\text{ as }n\rightarrow\infty\text{ in }\mathcal{V}.
\end{align*}
It remains to show $u=A_{\tau_N}^{-1}(f)$. Since $A_{\tau_N}(u_n)=f_n$ for all $n\in\mathbb{N}$, it follows that $A_{\tau_N}(u_n)\rightarrow f$ in $\mathcal{V}'$ for $n\rightarrow\infty$. We show that $f=A_{\tau_N}(u)$. From the convergence results for $(f_n)_{n\in\mathbb{N}}$ and $(u_n)_{n\in\mathbb{N}}$ respectively, it follows that 
\[\lim_{n\rightarrow\infty}\langle A_{\tau_N}(u_n),u_n-u\rangle_{\mathcal{V}',\mathcal{V}}=0\]
and combining the above equation with the monotonicity of $A_{\tau_N}$ we get
\begin{align*}
\lambda\langle f,v\rangle_{\mathcal{V}',\mathcal{V}}=\lim_{n\rightarrow\infty} \langle A_{\tau_N}(u_n),\lambda v\rangle_{\mathcal{V}',\mathcal{V}}\geq \langle A(u-\lambda v),\lambda v\rangle_{\mathcal{V}',\mathcal{V}}
\end{align*}
for any $v\in \mathcal{V}$ and any $\lambda\in\mathbb{R}$ and the assertion follows using the hemicontinuity of $A_{\tau_N}$.
Therefore 
\begin{align*}
A_{\tau_N}^{-1}(f_n)=u_n\rightharpoonup u=A_{\tau_N}^{-1}(f)\text{ as }n\rightarrow\infty\text{ in }\mathcal{V},
\end{align*}
hence $A_{\tau_N}^{-1}:\mathcal{V}'\rightarrow \mathcal{V}$ is demi-continuous.
\end{proof}

\subsubsection{A priori estimates and convergence results}
\begin{lemma}\label{210521_01}
For an $\mathcal{F}_0$-measurable random variable $u^0=u_0\in L^2(\Omega\times\mathbb{R}^d)$ and $k\in \{0,\ldots,N-1\}$ let $u^{k+1}$ be a solution to \eqref{timediscretqu}. With the convention $t_{-1}:=-\tau_N$ and $\Phi(t)=0$ a.s. in $\Omega$ for all $t<0$ we have
\begin{align}\label{200908_02}
\halbe \erwb\|u^{k+1}\|_2^2-\|u^k\|_2^2\erwe+\tau_N\erwb\|\nabla u^{k+1}\|_p^p\erwe\leq\halbe\mathbb{E}\left[\int_{t_{k-1}}^{t_k}\Vert\Phi(s)\Vert^2_{\operatorname{HS}}\,ds\right]
\end{align}
and, for $n\in\{0,\dots,N-1\}$
\begin{align}\label{energyestimation1}
\begin{split}
\halbe \erwb\|u^{n+1}\|_2^2-\|u_0\|_2^2\erwe+\tau_N\erwb\sum_{k=0}^n\|\nabla u^{k+1}\|_p^p\erwe&\leq\frac{1}{2}\|\Phi\|_{L^2(\Omega\times (0,T);HS(L^2(\mathbb{R}^d)))}^2\\
\end{split}
\end{align}
\end{lemma}
\begin{proof}
Since \eqref{energyestimation1} follows from \eqref{200908_02} by summing from $0$ to $n$, we have to show \eqref{200908_02}. For $k\in\{0,\dots,N-1\}$, we use $u^{k+1}$ as a test function in \eqref{timediscretqu} and obtain
\begin{align}\label{testtdeukp1}
\int_{\re^d}(u^{k+1}-u^k)u^{k+1}\,dx+\tau_N\|\nabla u^{k+1}\|_p^p=\int_{\re^d}\Phi^k\triangle_{k+1}Wu^{k+1}\,dx\text{ a.e. in }\Omega.
\end{align}
We study the different terms separately. It is obvious that
\begin{align*}
\int_{\re^d}(u^{k+1}-u^k)u^{k+1}\,dx&=\|u^{k+1}\|_2^2-\int_{\re^d}u^{k}u^{k+1}\,dx\\
&=\halbe\|u^{k+1}\|_2^2-\halbe\|u^{k}\|_2^2+\halbe\|u^{k+1}-u^k\|_2^2\text{ a.e. in }\Omega.
\end{align*}
On the right side of \eqref{testtdeukp1} we write $\triangle_{k+1}Wu^{k+1}=\triangle_{k+1}Wu^{k}+\triangle_{k+1}W(u^{k+1}-u^k)$ and recall that
\begin{align*}
\erwb\int_{\re^d}\Phi^k\triangle_{k+1}Wu^k\,dx\erwe&=\erwb\erwb\int_{\re^d}\Phi^k\triangle_{k+1}Wu^k\,dx\erwe\bigg|\mathcal{F}_{t_k}\erwe\\
&=\erwb\int_{\re^d}\Phi^k u^k\erwb \triangle_{k+1}W|\mathcal{F}_{t_k}\erwe\,dx\erwe\\
&=0,
\end{align*}
because $\Phi^k$ and $u^k$ are $\mathcal{F}_{t_k}$-measurable and $\triangle_{k+1}W$ is independent from $\mathcal{F}_{t_k}$. Moreover we get by Hölder's inequality and Young's inequality
\begin{align*}
&\erwb\int_{\re^d}\Phi^k\triangle_{k+1}W(u^{k+1}-u^k)\,dx\erwe\leq\erwb\|\Phi^k\triangle_{k+1}W\|_2\|u^{k+1}-u^k\|_2\erwe\\
&\leq\erwb\halbe\|\Phi^k\triangle_{k+1}W\|_2^2\erwe+\erwb\halbe\|u^{k+1}-u^k\|_2^2\erwe.
\end{align*}
From the It\^{o} isometry it follows that
\begin{align*}
\erwb\|\Phi^k\triangle_{k+1}W\|_2^2\erwe&=\erwb\left\|\int_{t_k}^{t_{k+1}}\Phi^k\,dW_s\right\|_2^2\erwe=\erwb\int_{t_k}^{t_{k+1}}\|\Phi^k\|_\HS^2\,ds\erwe\\
&=\erwb\tau_N\|\Phi^k\|_\HS^2\erwe\leq \mathbb{E}\left[\int_{t_{k-1}}^{t_k}\Vert\Phi(s)\Vert^2_{\operatorname{HS}}\,ds\right].
\end{align*}
Together we get
\begin{align*}
&\erwb\int_{\re^d}\Phi^k\triangle_{k+1}W(u^{k+1}-u^k)\,dx\erwe\\
&\leq\halbe\mathbb{E}\left[\int_{t_{k-1}}^{t_k}\Vert\Phi(s)\Vert^2_{\operatorname{HS}}\,ds\right]+\halbe\erwb\|u^{k+1}-u^k\|_2^2\erwe.
\end{align*}
Plugging the above inequalities into \eqref{testtdeukp1} we get
\begin{align*}
&\halbe \erwb \|u^{k+1}\|_2^2-\|u^k\|_2^2\erwe+\halbe\erwb\|u^{k+1}-u^k\|_2^2\erwe+\tau_N\erwb\|\nabla u^{k+1}\|_p^p\erwe\\
&=\erwb\int_{\re^d}\Phi^k\triangle_{k+1}W(u^{k+1}-u^k)\,dx\erwe-\erwb\int_{\re^d}\Phi^k\triangle_{k+1}W u^k\,dx\erwe\\
&\leq\halbe\mathbb{E}\left[\int_{t_{k-1}}^{t_k}\Vert\Phi(s)\Vert^2_{\operatorname{HS}}\,ds\right]+\halbe\erwb\|u^{k+1}-u^k\|_2^2\erwe.
\end{align*}
This yields the assertions.
\end{proof}
In the following, we want to integrate \eqref{timediscretqu} over $\Omega$ and $(0,T)$ and pass to the limit with the discretization parameter $N\rightarrow\infty$. To do so, we will introduce some notation and state some convergence results.
\begin{definition}
For $N\in\na$ and $t\in[0,T]$ we define
\begin{align*}
u_N^r(t)&:=\sum_{k=0}^{N-1}u^{k+1}\chi_{[t_k,t_{k+1})}(t),\ u_N^r(T)=u^N\\
u_N^l(t)&:=\sum_{k=0}^{N-1}u^{k}\chi_{(t_k,t_{k+1}]}(t),\ u_N^l(0)=u_0\\
\hat{u}_N(t)&:=\sum_{k=0}^{N-1}\left(\frac{u^{k+1}-u^k}{\tau_N}(t-t_k)+u^k\right)\chitk,\ \hat{u}_N(T)=u^N\\
\Phi_N(t)&:=\sum_{k=0}^{N-1}\Phi^k\chitk(t),\ \Phi_N(T)=\Phi^N
\end{align*}
and we recall that $\Phi^0=0$, $\Phi^k=\frac{1}{\tau_N}\int_{t_{k-1}}^{t_k}\Phi(s)\,ds$ for $k\in\{1,\dots,N\}$.
\begin{align*}
M_N(t)&:=\int_0^t\Phi_N\,dW_s\\
\hat{M}_N(t)&:=\sum_{k=0}^{N-1}\left(\frac{M_N(t_{k+1})-M_N(t_k)}{\tau_N}(t-t_k)+M_N(t_k)\right)\chitk,\\ 
\hat{M}_N(T)&=M_N(T).
\end{align*}
\end{definition}
From \cite[Lemma 12, p.52]{Simon07} it follows that
\[\Phi_N\rightarrow\Phi \ \text{for} \ N\rightarrow\infty \ \text{in} \ L^2(\Omega\times (0,T);\operatorname{HS}(L^2(\mathbb{R}^d))).\] 
In the following, we will investigate the boundedness and convergence of some of these functions interpreted as sequences for $N\in\na$.

\begin{lemma}\label{boundednesses}
There holds
\begin{itemize}
\item[(i)] $(u_N^r)_{N\in\na}$ and $(u_N^l)_{N\in\na}$ are bounded in $L^2(\Omega\times(0,T);L^2(\re^d))$.
\item[(ii)] $(\nabla u_N^r)_{N\in\na}$ 
is bounded in $L^p(\Omega\times (0,T); L^p(\re^d)^d)$.
\item[(iii)] $(\hat{u}_N)_{N\in\na}$ is bounded in $L^2(\Omega\times(0,T);L^2(\re^d))$.
\end{itemize}
In particular, passing to a not relabeled subsequence if necessary, for $N\rightarrow\infty$, 
\begin{itemize}
\item $(u_N^l)_{N\in\na}$ and $(\hat{u}_N)_{N\in\na}$ converge weakly in $L^2(\Omega\times(0,T);L^2(\re^d))$,
\item $(u_N^r)_{N\in\na}$ converges weakly in $L^2(\Omega\times (0,T);L^2(\mathbb{R}^d))$ and $(\nabla u_N^r)_{N\in\na}$ converges weakly in $L^p(\Omega\times (0,T);L^p(\mathbb{R}^d)^d)$, thus $(u_N^r)_{N\in\na}$ converges weakly in $L^q(\Omega\times (0,T);\mathcal{V})$ with $q:=\min\{p,2\}$
\end{itemize} 
\end{lemma}

\begin{proof}
(i): At first we see, that there holds for $N\in\na$
\begin{align*}
\erwb\int_0^T\|u_N^r(t)\|_2^2\,dt\erwe&=\erwb\sum_{n=0}^{N-1}\int_{t_n}^{t_{n+1}}\left\|\sum_{k=0}^{N-1}u^{k+1}\chitk(t)\right\|_2^2\,dt\erwe\\
&=\tau_N\sum_{n=0}^{N-1}\erwb\|u^{n+1}\|_2^2\erwe\\
\end{align*}
Using \eqref{energyestimation1}, we obtain
\begin{align*}
\erwb\int_0^T\|u_N^r(t)\|_2^2\,dt\erwe&\leq T\left(\erwb \|u_0\|_2^2\erwe +\Vert \Phi\Vert^2_{L^2(\Omega\times (0,T);HS(L^2(\mathbb{R}^d)))}\right).
\end{align*}
With the same arguments it is possible to show that $(u_N^l)_{N\in\na}$ is bounded in $L^2(\otwr{2})$.\\
(ii): By the definition of $u_N^r$ we know that
$$\nabla u_N^r=\sum_{k=0}^{N-1}\nabla u^{k+1}\chitk.$$
Applying \eqref{energyestimation1}, we obtain for $N\in\na$
\begin{align*}
\erwb\int_0^T\|\nabla u_N^r(t)\|_p^p\,dt\erwe
&=\tau_N\erwb\sum_{n=0}^{N-1}\|\nabla u^{n+1}\|_p^p\erwe\\
&\leq\frac{T}{2}\Vert \Phi\Vert^2_{L^2(\Omega\times (0,T);HS(L^2(\mathbb{R}^d)))}+\halbe\erwb\|u_0\|_2^2\erwe.
\end{align*}
(iii): For $N\in\na$ there holds
\begin{align*}
\erwb\int_0^T\|\hat{u}_N(t)\|_2^2\,dt\erwe
&=\erwb\sum_{n=0}^{N-1}\int_{t_n}^{t_{n+1}}\left\|\frac{u^{n+1}-u^n}{\tau_N}(t-t_n)+u^n\right\|_2^2\,dt\erwe\\
&\leq\erwb\sum_{n=0}^{N-1}\int_{t_n}^{t_{n+1}}\left(\frac{|t-t_n|}{\tau_N}\|u^{n+1}-u^n\|_2+\|u^n\|_2\right)^2\,dt\erwe\\
&\leq\erwb\sum_{n=0}^{N-1}\tau_N(\|u^{n+1}\|_2+2\|u^n\|_2)^2\,dt\erwe\\
&\leq 4\tau_N\erwb\sum_{n=0}^{N-1}\|u^{n+1}\|_2^2+4\|u^n\|_2^2\erwe.
\end{align*}
From \eqref{energyestimation1} it follows that $\sum_{n=0}^{N-1}\erwb\|u^{n+1}\|_2^2\erwe$ is bounded with a constant independent from $N$, and therefore the above expression is also bounded.
\end{proof}

Lemma \ref{boundednesses} implies that $(u_N^r)_{N\in\mathbb{N}}$ and $(u_N^l)_{N\in\mathbb{N}}$ converge weakly in $L^2(\Omega\times(0,T);L^2(\re^d))$ and $(\hat{u}_N)_{N\in\mathbb{N}}$ converges weakly in $L^2(\Omega\times(0,T);L^2(\re^d))$ as $N\rightarrow\infty$, passing to not relabeled subsequences if necessary. It would be desirable that the functions converge to the same limit.

\begin{lemma}
Passing to not relabeled subsequences if necessary, the sequences $(u_N^r)_{N\in\mathbb{N}}$, $(u_N^l)_{N\in\mathbb{N}}$ and $(\hat{u}_N)_{N\in\mathbb{N}}$ converge weakly in $L^2(\Omega\times(0,T);L^2(\re^d))$ to the same limit $u\in L^2(\Omega\times (0,T);L^2(\mathbb{R}^d))\cap L^q(\Omega\times (0,T);\mathcal{V})$ with $q=\min\{2,p\}$.
\end{lemma}

\begin{proof}
The argumentation is based on the estimate
\begin{align}\label{sumukp1-uk}
\erwb\sum_{k=0}^{N-1}\|u^{k+1}-u^k\|_2^2\erwe\leq C
\end{align}
for a constant $C\geq 0$ not depending on $N\in\mathbb{N}$. This can be derived from \eqref{testtdeukp1}. To obtain the desired result, we estimate the right-hand side of \eqref{testtdeukp1} using H\"older and Young inequalities in the following way:
\begin{align*}
&\erwb \int_{\re^d}\Phi^k\triangle_{k+1}W(u^{k+1}-u^k)\,dx\erwe\leq\erwb\|\Phi^k\triangle_{k+1}W\|_2\|u^{k+1}-u^k\|_2\erwe\\
&\leq\halbe \erwb 4\|\Phi^k\triangle_{k+1}W\|_2^2\erwe+\halbe\erwb\frac{1}{4}\|u^{k+1}-u^k\|_2^2\erwe.
\end{align*}
Therefore,
\begin{align*}
\frac{3}{8}\erwb\|u^{k+1}-u^k\|_2^2\erwe&\leq 2\erwb\|\Phi^k\triangle_{k+1}W\|_2^2\erwe+\halbe\erwb\|u^k\|_2^2\erwe-\halbe\erwb\|u^{k+1}\|_2^2\erwe\\
&\leq 2\mathbb{E}\left[\int_{t_{k-1}}^{t_k}\Vert\Phi(s)\Vert^2_{\operatorname{HS}}\,ds\right]+\halbe\erwb\|u^k\|_2^2\erwe-\halbe\erwb\|u^{k+1}\|_2^2\erwe.
\end{align*}
Summing the above inequalities over $k=0,\ldots,N-1$ we get
\begin{align*}
\erwb\sum_{k=0}^{N-1}\|u^{k+1}-u^k\|_2^2\erwe&\leq\frac{16}{3}\mathbb{E}\left[\int_{0}^{T}\Vert\Phi(s)\Vert^2_{\operatorname{HS}}\,ds\right]+\frac{4}{3}\erwb\|u_0\|_2^2\erwe=:C.
\end{align*}
Thus, for $N\in\na$, we have
\begin{align*}
\erwb\int_0^T\|u_N^r(t)-u_N^l(t)\|_2^2\,dt\erwe&=\erwb\sum_{k=0}^{N-1}\int_{t_k}^{t_{k+1}}\|u^{k+1}-u^k\|_2^2\,dt\erwe\\
&\leq\tau_N \erwb\sum_{k=0}^{N-1}\|u^{k+1}-u^k\|_2^2\erwe\\
&\leq\tau_N C
\end{align*}
Hence, $\lim_{N\rightarrow\infty}\erwb\int_0^T\|u_N^r(t)-u_N^l(t)\|_2^2\,dt\erwe=0$ and from the weak lower semicontinuity of the norm it follows that $(u_N^l)_{N\in\mathbb{N}}$ and $(u_N^r)_{N\in\mathbb{N}}$ converge weakly to the same limit in $u\in L^2(\Omega\times(0,T);L^2(\re^d))$ as $N\rightarrow\infty$. With the additional information on the weak convergence of $(u_N^r)_{N\in\mathbb{N}}$ in $L^q(\Omega\times (0,T);\mathcal{V})$ (see Lemma \ref{boundednesses}), it follows that $u$ is also in $L^q(\Omega\times (0,T);\mathcal{V})$. Furthermore,
\begin{align*}
\erwb\int_0^T\|\hat{u}_N(t)-u_N^l(t)\|_2^2\,dt\erwe&=\erwb\sum_{k=0}^{N-1}\int_{t_k}^{t_{k+1}}\left\|\frac{u^{k+1}-u^k}{\tau_N}(t-t_k)\right\|_2^2\,dt\erwe\\
&\leq \tau_N\erwb\sum_{k=0}^{N-1}\|u^{k+1}-u^k\|_2^2\erwe\leq \tau_NC
\end{align*}
and therefore $\lim_{N\rightarrow\infty}\hat{u}_N=u$ weakly in $L^2(\Omega\times(0,T);L^2(\re^d))$.
\end{proof}

\begin{lemma}\label{200909_lem01}
The sequence $(M_N)_{N\in\na}$ converges strongly to $M:=\int_0^\cdot \Phi\,dW_s$ in $L^2(\Omega\times(0,T);L^2(\re^d))$ as $N\rightarrow\infty$.
\end{lemma}

\begin{proof}
We recall that $(\Phi_N)_{N\in\mathbb{N}}$ converges strongly to $\Phi$ in $L^2(\Omega\times (0,T);HS(L^2(\mathbb{R}^d)))$ for $N\rightarrow\infty$.
From the It\^{o} isometry, we get
\begin{align*}
\erwb\int_0^T\|M_N(t)-M(t)\|_2^2\,dt\erwe&=\erwb\int_0^T\left\|\int_0^t\Phi_N(s)-\Phi(s)\,dW_s\right\|_2^2\,dt\erwe\\
&\leq T\erwb\int_0^T\|\Phi_N(s)-\Phi(s)\|_\HS^2\,ds\erwe\\
\end{align*}
and therefore the assertion follows.
\end{proof}
\begin{lemma}\label{convMN-hatM_N}
We have
\begin{align}\label{convMN-hatMN}
\lim_{N\rightarrow\infty}\erwb\int_0^T\|M_N(t)-\hat{M}_N(t)\|^2_2\,dt\erwe= 0.
\end{align}
In particular, $(\hat{M}_N)_{N\in\na}$ converges strongly to $M:=\int_0^\cdot \Phi\,dW_s$ in $L^2(\Omega\times(0,T);L^2(\re^d))$ as $N\rightarrow\infty$.
\end{lemma}

\begin{proof}
For any $t\in[t_k,t_{k+1})$ and $k\in\{0,\ldots,N-1\}$ we have
\begin{align*}
\|M_N(t)-\hat{M}_N(t)\|_2^2
&=\left\|\int_0^t\Phi_N\,dW_s-\frac{t-t_k}{\tau_N}\int_{t_k}^{t_{k+1}}\Phi_N\,dW_s-\int_0^{t_k}\Phi_N\,dW_s\right\|_2^2\\
&\leq 2\left(\left\|\int_{t_k}^t\Phi^k\,dW_s\right\|_2^2
+\left\|\int_{t_k}^{t_{k+1}}\Phi^k\,dW_s\right\|_2^2\right).
\end{align*}
Using the It\^{o} isometry and then H\"older inequality, it follows that
\begin{align*}
&\int_0^T\mathbb{E}\left[\|M_N(t)-\hat{M}_N(t)\|_2^2\right]\,dt
=\sum_{k=0}^{N-1}\int_{t_k}^{t_{k+1}}\erwb \|M_N(t)-\hat{M}_N(t)\|_2^2\erwe\,dt\\
&\leq 2\sum_{k=0}^{N-1}\int_{t_k}^{t_{k+1}}\left(\int_{t_k}^t \erwb\|\Phi^k\|^2_{\HS}\erwe\,ds+\int_{t_k}^{t_{k+1}}\erwb\|\Phi^k\|^2_{\HS}\erwe\,ds\right)\,dt\\
&\leq 4(\tau_N)^2\sum_{k=0}^{N-1}\erwb\|\Phi^k\|^2_{\HS}\erwe\\
&\leq 4\tau_N\Vert \Phi\Vert_{L^2(\Omega\times (0,T);HS(L^2(\mathbb{R}^d)))},
\end{align*}
hence \eqref{convMN-hatMN} holds true. The convergence result is a direct consequence of Lemma \ref{200909_lem01}.
\end{proof}

In the next steps, we characterize $\frac{\partial}{\partial t}(\hat{u}_N-\hat{M}_N)$ and its weak limit. 
\begin{lemma}\label{200909_lem02}
Passing to a not relabeled subsequence if necessary, $(\frac{\partial}{\partial t}(\hat{u}_N(t)-\hat{M}_N(t)))_{N\in\mathbb{N}}$ weakly converges to $\frac{\partial}{\partial t}(u-\int_0^{\cdot} \Phi\,dW_s)$ in $L^{p'}(\Omega\times (0,T);\mathcal{V}')$.
\end{lemma}
\begin{proof}
Let $k\in\{0,\dots,N-1\}$. For $t\in[t_k,t_{k+1})$ there holds
\begin{align*}
\frac{\partial}{\partial t}\left(\hat{u}_N(t)-\hat{M}_N(t)\right)
&=\frac{u^{k+1}-u^k}{\tau_N}-\frac{M_N(t_{k+1})-M_N(t_k)}{\tau_N}.
\end{align*}
From \eqref{timediscretqu} it follows that
\begin{align*}
u^{k+1}-u^k-\tau_N\Delta_p(u^{k+1})&=\Phi^k\triangle_{k+1}W=M_N(t_{k+1})-M_N(t_k).
\end{align*}
Thus, \eqref{timediscretqu} can be rewritten as
\begin{align*}
\frac{\partial}{\partial t}\left(\hat{u}_N(t)-\hat{M}_N(t)\right)=\Delta_p(u^{k+1})\text{ for }t\in[t_k,t_{k+1}).
\end{align*}
By H\"older's inequality, for $t\in[t_k,t_{k+1})$ it follows that 
\begin{align*}
&\left\|\frac{\partial}{\partial t}\left(\hat{u}_N(t)-\hat{M}_N(t)\right)\right\|_{\mathcal{V}'}\leq\sup_{\varphi\in \mathcal{V},\|\varphi\|_{\mathcal{V}}\leq1}\int_{\re^d}|\nabla u^{k+1}|^{p-1}|\nabla\varphi|\,dx\\
&\leq\|\nabla u^{k+1}\|_p^{p-1}.
\end{align*}
According to this we get
\begin{align*}
\erwb\int_0^T\left\|\frac{\partial}{\partial t}\left(\hat{u}_N(t)-\hat{M}_N(t)\right)\right\|_{\mathcal{V}'}^{p'}\,dt\erwe
&\leq\erwb\sum_{n=0}^{N-1}\int_{t_n}^{t_{n+1}}\|\nabla u^{n+1}\|_p^p\,dt\erwe\\
&=\erwb\int_0^T\|\nabla u_N^r\|_p^p\,dt\erwe.
\end{align*}
Now, from Lemma \ref{boundednesses} (ii), it follows that $(\frac{\partial}{\partial t}(\hat{u}_N-\hat{M}_N))_{N\in\mathbb{N}}$ is bounded in $L^{p'}(\Omega\times(0,T);\mathcal{V}')$ and therefore we may extract a not relabeled subsequence having a weak limit in $L^{p'}(\Omega\times(0,T);\mathcal{V}')$ for $N\rightarrow\infty$. Using integration by parts and the weak convergence of $(\hat{u}_N)_{N\in\mathbb{N}}$ and $(\hat{M}_N)_{N\in\mathbb{N}}$ we may conclude that the limit element is $\frac{\partial}{\partial t}(u-\int_0^{\cdot} \Phi\,dW_s)$.
\end{proof}
\subsubsection{Existence and uniqueness result}
\begin{proposition}
For $q=\min\{2,p\}$, let 
\[u\in L^q(\Omega\times (0,T);\mathcal{V})\cap L^2(\Omega\times (0,T);L^2(\mathbb{R}^d))\] 
be the weak limit of $(\hat{u}_N)_{N\in\mathbb{N}}$ as $N\rightarrow\infty$ in $L^2(\otwr{2})$. Then $u$ is a solution to \eqref{spdeadditiv}.
\end{proposition}

\begin{proof}
From \eqref{timediscretqu} it follows that
\begin{align*}
\frac{\partial}{\partial t}\left(\hat{u}_N(t)-\hat{M}_N(t)\right)-\Delta_p(u^r_N)=0
\end{align*}
in $L^{p'}(\Omega\times (0,T);\mathcal{V}')$.
For $\psi\in\mathcal{D}(\re^d)$, $\rho\in\mathcal{D}(0,T)$ and $A\in\mathcal{F}$ we get
\begin{align}\label{testedtde}
\begin{split}
&-\int_A\int_0^T\int_{\re^d}\left(\hat{u}_N(t)-\hat{M}_N(t)\right)\rho_t(t)\psi\,dx\,dt\,dP\\
&+\int_A\int_0^T\int_{\re^d}|\nabla u_N^r(t)|^{p-2}\nabla u_N^r(t)\cdot\nabla\psi\rho(t)\,dx\,dt\,dP\\
&=0.
\end{split}
\end{align}
From Lemma \ref{boundednesses} (ii) it follows that $(\nabla u_N^r)_{N\in\mathbb{N}}$ is bounded in \\$L^p(\Omega\times(0,T);L^p(\re^d)^d)$. Because of that, $(|\nabla u_N^r|^{p-2}\nabla u_N^r)_{N\in\mathbb{N}}$ is bounded in $L^{p'}(\Omega\times(0,T);L^{p'}(\re^d)^d)$. Hence there exists a subsequence, still denoted by $(|\nabla u_N^r|^{p-2}\nabla u_N^r)_{N\in\mathbb{N}}$, such that
\begin{align*}
|\nabla u_N^r|^{p-2}\nabla u_N^r\rightharpoonup G\text{ as }N\rightarrow\infty\text{ in }L^{p'}(\Omega\times(0,T);L^{p'}(\re^d )^d)
\end{align*}
for an element $G\in L^{p'}(\Omega\times(0,T);L^{p'}(\re^d)^d)$. Using Lemma \ref{200909_lem02} for the passage to the limit on the left-hand side of \eqref{testedtde} it follows that 

\begin{align*}
-\int_A\int_0^T\int_{\re^d}(u(t)-M(t))\rho_t(t)\psi\,dx\,dt\,dP=-\int_A\int_0^T\int_{\re^d} G\cdot\nabla\psi\rho(t)\,dx\,dt\,dP.
\end{align*}
From the above equation it follows that
\begin{align}\label{200909_01}
\frac{d}{dt}\left(u-\int_0^{\cdot} \Phi\,dW_s\right)=\diver G
\end{align}
in $\mathcal{D}'((0,T)\times\re^d)$, a.s. in $\Omega$.
Because $G$ is an element of $L^{p'}(\Omega\times (0,T);L^{p'}(\re^d)^d)$, equation \eqref{200909_01} also holds in $L^{p'}(\Omega\times(0,T);\mathcal{V}')$. From \cite[Lemma 7.1, p.202]{Roubicek} it follows that
\begin{align*}
u-\int_0^{\cdot} \Phi\,dW_s\in \mathcal{C}([0,T];\mathcal{V}')\text{ a.s. in }\Omega.
\end{align*}
Taking into account that $\int_0^{\cdot} \Phi\,dW_s$ has paths in $\mathcal{C}([0,T];L^2(\mathbb{R}^d))$ a.s. in $\Omega$, it follows that $u\in \mathcal{C}([0,T];\mathcal{V}')$ a.s. in $\Omega$ and 
\begin{align}\label{210520_01}
u(t)-u_0+\int_0^t \diver G\,ds=\int_0^t \Phi(s)\,dW_s.
\end{align}
holds true for all $t\in [0,T]$, a.s. in $\Omega$. From \eqref{210520_01} and $u\in L^2(\Omega\times (0,T);L^2(\mathbb{R}^d))$ it follows that the It\^{o} formula (\cite{LiuRock}, Theorem 4.2.5) applies and therefore $u\in \mathcal{C}([0,T];L^2(\mathbb{R}^d))$ a.s. in $\Omega$. The proof is complete as soon as we show $G=|\nabla u|^{p-2}\nabla u$. This will be done in the following lemma.
\end{proof}

\begin{lemma}\label{G=nablasu}
There holds
\begin{align*}
G=|\nabla u|^{p-2}\nabla u.
\end{align*}
in $L^{p'}(\Omega\times(0,T)\times \mathbb{R}^d)^d$.
\end{lemma}

\begin{proof}
We recall \eqref{200908_02}, sum over $k=0,\ldots, N-1$ and use that $\hat{u}_N(T)=u^N$ to obtain
\begin{align}\label{itofomapprox}
\begin{split}
\frac{1}{2}\erwb\Vert u_0\Vert^2_2\erwe \geq
&\erwb\int_0^T\int_{\mathbb{R}^d} |\nabla u^{r}_N|^{p-2}\nabla u^{r}_N\cdot\nabla u^{r}_N \,dx \,dt\erwe\\
&+\frac{1}{2} \erwb\Vert \hat{u}_N(T)\Vert^2_2\erwe
-\frac{1}{2}\erwb\int_0^T \Vert \Phi_N\Vert^2_{HS} \, dt\erwe
\end{split}
\end{align}
It\^{o} formula in \eqref{210520_01} with $t=T$ yields
\begin{align}\label{itoformconvapprox}
\halbe\erwb\|u_0\|_2^2\erwe=\erwb\int_0^T\int_{\re^d} G\cdot \nabla u\,dx\,dt\erwe+\halbe\erwb\|u(T)\|_2^2\erwe-\halbe\erwb\int_0^T\|\Phi\|_\HS^2\,dt\erwe.
\end{align}
Subtracting \eqref{itoformconvapprox} from \eqref{itofomapprox}, we get
\begin{align}\label{diffitosforapandconvap}
\begin{split}
0&\geq\,\erwb\int_0^T\int_{\re^d}|\nabla u^{r}_N|^{p-2}\nabla u^{r}_N\cdot\nabla u^{r}_N \,dx \,dt\erwe-\erwb\int_0^T\int_{\re^d} G\cdot\nabla u\,dx\,dt\erwe\\
-&\halbe\erwb\int_0^T\|\Phi_N\|_\HS^2\,dt\erwe+\halbe\erwb\int_0^T\|\Phi\|_\HS^2\,dt\erwe +\frac{1}{2} \erwb\Vert \hat{u}_N(T)\Vert^2_2\erwe-\halbe\erwb\|u(T)\|_2^2\erwe.
\end{split}
\end{align}
Since $\Phi_N\rightarrow\Phi$ in $L^2(\Omega\times(0,T);\operatorname{HS}(L^2(\mathbb{R}^d))$, if we show that
\begin{align}\label{210521_02}
\liminf_{N\rightarrow\infty}\erwb\| \hat{u}_N(T)\|_2^2-\|u(T)\|_2^2\erwe\geq 0
\end{align}
holds, we are able to conclude from \eqref{diffitosforapandconvap}
\begin{align}\label{nablauepsinlpforG}
\limsup_{N\rightarrow\infty}\erwb\int_0^T\int_{\re^d}|\nabla u^{r}_N|^{p-2}\nabla u^{r}_N\cdot\nabla u^{r}_N \,dx \,dt\erwe \leq\erwb\int_0^T\int_{\re^d} G\cdot\nabla u\,dx\,ds\erwe.
\end{align}
From Lemma \ref{210521_01}, \eqref{energyestimation1} it follows that there exists a constant $C_1>0$ not depending on $N\in\mathbb{N}$ such that
\begin{align*}
\erwb\| \hat{u}_N (T)\|_2^2\erwe\leq C_1
\end{align*}
Because of that, we find a function $\chi\in L^2(\Omega\times\re^d)$ and a subsequence, still denoted by $(\hat{u}_N(T))_N$, such that
\begin{align}\label{convueps(t)}
\hat{u}_N(T)\rightharpoonup\chi\text{ as }N\rightarrow\infty \ \text{ in }L^2(\Omega\times\re^d).
\end{align}
We want to verify $u(T)=\chi$.\\
Therefore, we choose arbitrary functions $h\in C_c^\infty(\re^d)$, $g\in C^\infty([0,T])$, $A\in\mathcal{F}$ and compare the expressions 
\begin{align*}
&\int_A\left\langle\frac{d}{dt}\left(u(t)-\int_0^t\Phi\,dW_s\right),\ h(x)g(t)\right\rangle\,dP\\
\text{and}\quad &\int_A\left\langle\frac{d}{dt}\left(\hat{u}_N(t)-\int_0^t \Phi_N\,dW_s\right), \ h(x)g(t)\right \rangle\,dP
\end{align*}
where $\langle\cdot,\cdot\rangle$ denotes the duality pairing of $L^p(0,T;\mathcal{V})$ and $L^{p'}(0,T;\mathcal{V}')$. By partial integration, for the first term we get
\begin{align}\label{eqforEgeq01}
\begin{split}
&\int_A\left\langle\frac{d}{dt}\left(u(t)-\int_0^t\Phi\,dW_s\right),\ h(x)g(t)\right\rangle\,dP\\
=\,&-\int_A\int_0^T\int_{\re^d}\left(u(t)-\int_0^t \Phi\,dW_s\right)h(x)\partial_t g(t)\,dx\,dt\,dP\\
&+\int_A\int_{\re^d}\left(u(T)-\int_0^T \Phi\,dW_s\right)h(x)g(T)\,dx\,dP\\
&-\int_A\int_{\re^d} u_0 h(x)g(0)\,dx\,dP.
\end{split}
\end{align}
On the other hand, for any $N\in\mathbb{N}$ we have
\begin{align*}
&\int_A\left\langle\frac{d}{dt}\left(\hat{u}_N(t)-\int_0^t \Phi_N\,dW_s\right), \ h(x)g(t)\right \rangle\,dP\\
=\,&-\int_A\int_0^T\int_{\re^d}\left(\hat{u}_N(t)-\int_0^t \Phi_N\,dW_s\right)h(x)\partial_t g(t)\,dx\,dt\,dP\\
&+\int_A\int_{\re^d}\left(\hat{u}_N(T)-\int_0^T \Phi_N\,dW_s\right)h(x)g(T)\,dx\,dP\\
&-\int_A\int_{\re^d}u_0h(x)g(0)\,dx\,dP.
\end{align*}
Passage to the limit with $N\rightarrow\infty$ in the above equation together with the weak convergence \eqref{convueps(t)} yields
\begin{align}\label{eqforEgeq02}
\begin{split}
&\int_A\left\langle\frac{d}{dt}\left(u(t)-\int_0^t\Phi\,dW_s\right),\ h(x)g(t)\right\rangle\,dP\\
\,=&-\int_A\int_0^T\int_{\re^d}\left(u(t)-\int_0^t\Phi\,dW_s\right)h(x)\partial_tg(t)\,dx\,dt\,dP\\
&+\int_A\int_{\re^d}\left(\chi(T)-\int_0^T \Phi\,dW_s\right)h(x)g(T)\,dx\,dP\\&-\int_A\int_{\re^d}u_0 h(x)g(0)\,dx\,dP.
\end{split}
\end{align}
Subtracting \eqref{eqforEgeq01} from \eqref{eqforEgeq02} we get
\begin{align*}
0=\int_A\int_{\re^d}\left(u(T)-\chi(T)\right)h(x)g(T)\,dx\,dP.
\end{align*}
If we choose $g\in C^\infty([0,T])$ such that $g(T)=1$, we can conclude
\begin{align*}
u(T)=\chi(T)\text{ a.e. in }\re^d, \text{ a.s. in }\Omega.
\end{align*}
From the weak lower semicontinuity of the norm it now follows that
\begin{align*}
\liminf_{N\rightarrow\infty}\erwb\| \hat{u}_N(T)\|_2^2\erwe\geq\erwb\|u(T)\|_2^2\erwe
\end{align*}
and thereby \eqref{210521_02} and \eqref{nablauepsinlpforG}.
Let $\Lambda$ be an element of $L^p(\Omega\times(0,T)\times\re^d)^d$. 
Using the fundamental inequality \eqref{technicalineq} we get
\begin{align*}
0\leq\,&\limsup_{N\rightarrow\infty}\erwb\int_0^T \langle |\nabla u^{r}_N|^{p-2}\nabla u^{r}_N-|\Lambda|^{p-2}\Lambda,\nabla u_N^r-\Lambda\rangle_{p',p}\,ds\erwe\\
\leq\,&\erwb\int_0^T\int_{\re^d} G\cdot\nabla u\,dx\,ds\erwe-\erwb\int_0^T\int_{\re^d}|\Lambda|^{p-2}\Lambda\cdot\nabla u\,dx\,ds\erwe\\
&-\erwb\int_0^T\int_{\re^d}(G-|\Lambda|^{p-2}\Lambda)\cdot\Lambda\,dx\,ds\erwe\\
=\,&\erwb\int_0^T\int_{\re^d}(G-|\Lambda|^{p-2}\Lambda)\cdot(\nabla u-\Lambda)\,dx\,ds\erwe.
\end{align*}
Now, choosing
$\Lambda:=\nabla u+\lambda\varphi$
with $\lambda\in\re$ and $\varphi\in L^p(\Omega\times(0,T)\times\re^d)^d$, we get

\begin{align*}
G=|\nabla u|^{p-2}\nabla u.
\end{align*}
in $L^{p'}(\Omega\times (0,T)\times\mathbb{R}^d)^d$.
\end{proof}

\section{Multiplicative noise}

Up to now we considered an additive stochastic noise $\Phi$. By a fixed point argument, we can also find a solution to \eqref{stochplaplaceequation} for a multiplicative noise. \\
Let $B:\Omega\times (0,T)\times\mathbb{R}\rightarrow \HS(L^2(\re^d))$ satisfy conditions $(B1)-(B3)$.
We define the solution operator
\begin{align*}
\Pi: L^2\left(\Omega\times(0,T)\times\re^d,\mathcal{P}\right)&\longrightarrow L^2\left(\Omega\times(0,T)\times\re^d,\mathcal{P}\right)\\
\rho &\longmapsto u_\rho
\end{align*}
where $u_\rho$ is the unique solution to
\begin{align*}
du-\diver(\nablas{u})\,dt&=B(\cdot,\rho)\,dW_t &\text{in }\Omega\times(0,T)\times\re^d\\
u(0,\cdot)&=u_0 &\text{in }\Omega\times\re^d 
\end{align*}
and $L^2\left(\Omega\times(0,T)\times\re^d,\mathcal{P}\right)$ denotes the space of $dP\otimes dt\otimes dx$ equivalence classes of progressively measurable functions $\rho:\Omega\times (0,T)\times\mathbb{R}^d\rightarrow\mathbb{R}$ that are square integrable on $\Omega\times (0,T)\times\mathbb{R}^d$.
This map is well defined, because $B(\cdot,\rho)$ is an Hilbert-Schmidt operator which satisfies all the conditions we used to proof the existence and uniqueness for a solution to \eqref{spdeadditiv}.\\
We want to show, that $\Pi$ is a strict contraction in $L^2\left(\Omega\times(0,T)\times\re^d,\mathcal{P}\right)$. For that we choose arbitrary $\rho_1,\rho_2\in L^2(\Omega\times(0,T)\times\re^d,\mathcal{P})$. Then $u_1:=\Pi(\rho_1)$ and $u_2:=\Pi(\rho_2)$ satisfy
\begin{align*}
du_1-\Delta_p(u_1)\,dt&=B(\cdot,\rho_1)\,dW_t &&\text{in }\Omega\times(0,T)\times\re^d\\
du_2-\Delta_p(u_2)\,dt&=B(\cdot,\rho_2)\,dW_t &&\text{in }\Omega\times(0,T)\times\re^d.
\end{align*} 
This leads us to
\begin{align*}
d(u_1-u_2)-[\Delta_p(u_1)-\Delta_p(u_2)]\,dt=[B(\cdot,\rho_1)-B(\cdot,\rho_2)]\,dW_t.
\end{align*}
By the It\^{o} formula, \eqref{technicalineq} and $(B3)$ we get for $t\in(0,T)$
\begin{align*}
\frac{1}{2}\erwb\|u_1(t)-u_2(t)\|_2^2\erwe&\leq\frac{1}{2}\int_0^t\erwb\|B(\cdot,\rho_1)-B(\cdot,\rho_2)\|_{HS}^2\erwe\\&\leq\frac{1}{2}\int_0^t\erwb L^2\|\rho_1-\rho_2\|_2^2\erwe\,ds.
\end{align*}
In the following, we show that $\Pi$ is a strict contraction for a weighted norm in $L^2(\Omega\times(0,T)\times\re^d,\mathcal{P})$.
For $\alpha>L^2$ and $t\in[0,T]$ using $e^{-\alpha t}=\frac{d}{dt}(-\frac{1}{\alpha}e^{-\alpha t})$ and integration by parts, we obtain
\begin{align*}
\erwb\int_0^T e^{-\alpha t} \|u_1(t)-u_2(t)\|_{2}^2\,dt\erwe\leq\,&L^2 \erwb\int_0^T e^{-\alpha t}\int_0^t\|\rho_1-\rho_2\|_2^2\,ds\,dt\erwe\\
=\,&L^2\erwb\int_0^T e^{-\alpha T}\int_0^T\|\rho_1-\rho_2\|_2^2\,ds\erwe\\
&-L^2\erwb\int_0^T-\frac{1}{\alpha}e^{-\alpha t}\int_0^t\|\rho_1-\rho_2\|_2^2\,ds\,dt\erwe\\
\leq\,&\frac{L^2}{\alpha}\erwb\int_0^T e^{-\alpha t}\|\rho_1-\rho_2\|_2^2\,dt\erwe.
\end{align*}
So $\Pi$ is a strict contraction in $L^2\left(\Omega\times(0,T)\times\re^d;e^{-\alpha t} ;\mathcal{P}\right)$. Note that $e^{-\alpha t}$ is absolutely continuous and positive, so that the Borel sets are the same. Now, the Banach fixed point theorem implies the existence of a unique element $\tilde{u}$ in $L^2\left(\Omega\times(0,T)\times\re^d;e^{-\alpha t} ;\mathcal{P}\right)$ such that
$$\Phi(\tilde{u})=\tilde{u}.$$
Obviously, $\tilde{u}$ is in $L^2\left(\Omega\times(0,T)\times\re^d,\mathcal{P}\right)$ 
and the uniqueness remains, since the weighted norm is equivalent to $\|\cdot\|_{L^2(\otr)}$ on $L^2\left(\Omega\times(0,T)\times\re^d,\mathcal{P}\right)$.

\section{Appendix: An alternative method for a special case}

From now on, we assume that the exponent $1<p<\infty$ is such that the continuous embedding $\sobolevp\hookrightarrow L^2(\re^d)$ is valid. This is the case for space dimension $d=2$ if $1\leq p<2$ and as well for space dimension $d>2$ if $1\leq p<d$ is such that $p\in [\frac{2d}{d+2},2]$, for $p=d=2$ or $p=d=1$ and for $p>d$ and $1<p\leq 2$. We consider \ref{spdeadditiv} in the framework of the Gelfand triple
\begin{align*}
\sobolevp\hookrightarrow L^2(\re^d)\hookrightarrow\sobolevq.
\end{align*}
We fix an $\mathcal{F}_0$-measurable initial datum $u_0\in L^2(\Omega\times\mathbb{R}^d)$ and search for an $(\mathcal{F}_t)_{t\geq 0}$-adapted stochastic process $u\in L^2(\Omega;\mathcal{C}([0,T];L^2(\re^d)))$ such that $u\in L^p(\Omega; L^p(0,T;\sobolevp))$ and
\begin{align*}
u(t)-u_0-\int_0^t \diver(\nablas{u})\,ds=\int_0^t \Phi\,dW_s\text{ for }t\in [0,T] \text{ a.s. in }\Omega.
\end{align*}

In the following, we will propose a regularized equation such that the associated operator satisfies hemicontinuity, weak monotonicity, coercivity and boundedness in the form as it is necessary for \cite[p.91]{LiuRock}. Hence we find a unique solution to the regularized equation. Then we show that this solution converges to a solution of \eqref{spdeadditiv} as the regularization parameter goes to $0$. There we need that $\Phi$ is an additive noise to pass to the limit in the equation. 

Let $0<\eps<1$ be arbitrary. We consider the regularized problem
\begin{align}\tag{AP}\label{approxequadd}
\begin{aligned}
du_{\eps} -\Delta_p(u_{\eps})\,dt+\eps|u_{\eps}|^{p-2}u_{\eps}\,dt&=\Phi \,dW_t &\text{in }\Omega\times(0,T)\times\re^d\\
u_{\eps}(0,\cdot)&=u_0 &\text{in }\Omega\times\re^d.
\end{aligned}
\end{align}

\begin{lemma}
For an $\mathcal{F}_0$-measurable initial datum $u_0\in L^2(\Omega\times\mathbb{R}^d)$ and any $0<\eps<1$  the equation \eqref{approxequadd} has a unique solution 
\[u_\eps\in L^2(\Omega;\mathcal{C}([0,T];L^2(\re^d)))\cap L^p(\Omega;L^p(0,T;\sobolevp)).\]
\end{lemma}

\begin{proof}
We define the operator
\begin{align*}
A_\eps:\sobolevp&\longrightarrow\sobolevq\\
v&\longmapsto \Delta_p(v)-\eps|v|^{p-2}v.
\end{align*}
It is easy to proof that $A_{\varepsilon}$ satisfies the hemicontinuity, weak monotonicity, coercivity and boundedness conditions of \cite[p.91]{LiuRock} in the framework of the Gelfand triple
\begin{align*}
\sobolevp\hookrightarrow L^2(\re^d)\hookrightarrow\sobolevq.
\end{align*}
Applying \cite[p.91]{LiuRock} provides, that, for each $0<\eps<1$, there exists a unique $(\mathcal{F}_t)_{t\geq 0}$-adapted stochastic process 
\[u_\eps\in L^2(\Omega, \mathcal{C}([0,T];L^2(\re^d)))\cap L^p(\Omega;L^p(0,T;\sobolevp))\] 
such that
\begin{align*}
u_\eps(t)-u_0-\int_{0}^t\Delta_p(u_\eps)\,ds+\int_0^t\eps|u_\eps|^{p-2}u_\eps\,ds=\int_0^t \Phi\,dW_s
\end{align*}
holds true in $L^2(\mathbb{R}^d)$ for all $t\in [0,T]$ a.s. in $\Omega$, thus $u_{\varepsilon}$ is a solution to \eqref{approxequadd}.
\end{proof}

Our intention is to show that $(u_\eps)_{0<\eps<1}$ converges to a solution of \eqref{spdeadditiv} as $\eps\downarrow0$.
The It\^{o} formula (see \cite[p.91f.]{LiuRock}) yields that for all $t\in[0,T]$ and any $0<\eps<1$ we have
\begin{align}\label{itoueps}
\begin{split}
\frac{1}{2}\erwb\|u_\eps(t)\|_2^2\erwe=&\frac{1}{2}\erwb\|u_0\|_2^2\erwe-\int_0^t \erwb\|\nabla\ueps\|_p^p\erwe\,ds-\eps\int_0^t\erwb\|\ueps\|_p^p\erwe \,ds\\
&+\frac{1}{2}\int_0^t\|\Phi\|_\HS^2\,ds.
\end{split}
\end{align}
If we use the boundedness of $\Phi$, from \eqref{itoueps} it follows that for any $t\in [0,T]$
\begin{align*}
\frac{1}{2}\erwb\|u_\eps(t)\|_2^2\erwe&\leq\frac{1}{2}\erwb\|u_0\|_2^2\erwe+\frac{\mathrm{C}_{\Phi}}{2}=:\frac{C_1}{2}.
\end{align*}
Now we get for all $0<\eps<1$
\begin{align}\label{uepsboundinl2}
\|u_\eps\|_{L^2(\Omega\times(0,T)\times\re^d)}^2\leq T\sup_{t\in(0,T)}\erwb\|u_\eps(t)\|_2^2\erwe\leq T C_1.
\end{align}
Since $L^2(\Omega\times(0,T)\times\re^d)$ is reflexive, there exists a weakly convergent subsequence, still denoted by $(u_\eps)_{0<\varepsilon<1}$ and an element $u\in L^2(\Omega\times(0,T)\times\re^d)$,  such that
\begin{align}\label{convuepsinl2}
u_\eps\rightharpoonup u\text{ as }\eps\downarrow0\text{ in }L^2(\Omega\times(0,T)\times\re^d).
\end{align}

\begin{proposition}\label{200904_prop1}
The element $u$ is the unique solution to \eqref{spdeadditiv}.
\end{proposition}

\begin{proof}
We consider the terms $\eps|u_\eps|^{p-2}u_\eps$ and $\Delta_p(\ueps)$ separately. Let's begin with the term $\eps|u_\eps|^{p-2}\ueps$ which we have added to get a coercive operator. Equation \eqref{itoueps} provides for $0<\eps<1$
\begin{align*}
\eps\int_0^T\erwb\|u_\eps\|_p^p\erwe\,ds=\,&\frac{1}{2}\erwb\|u_0\|_2^2\erwe-\frac{1}{2}\erwb\|u_\eps(T)\|_2^2\erwe-\int_0^T\erwb\|\nabla u_\eps\|_p^p\erwe\,ds\\
&+\frac{1}{2}\int_0^T\|\Phi\|_\HS^2\,ds\\
\leq\,&\frac{1}{2}\erwb\|u_0\|_2^2\erwe+\frac{\mathrm{C}_{\Phi}}{2}=\frac{C_1}{2}.
\end{align*}
Hence we have for $0<\eps<1$
\begin{align*}
\eps^{\frac{1}{p'}}\||u_\eps|^{p-2}\ueps\|_{L^{p'}(\otr)}&=\eps^{\frac{1}{p'}}\|u_\eps\|_{L^p(\otr)}^{p-1}\\
&=\left(\eps\|u_\eps\|_{L^p(\otr)}^p \right)^{\frac{p-1}{p}}\\
&\leq \left(\frac{C_1}{2}\right)^{\frac{p-1}{p}}.
\end{align*}
By Hölder's inequality this implies for any $\varphi\in L^p(\Omega;L^p(0,T;\sobolevp))$
\begin{align*}
\langle\eps|u_\eps|^{p-2}u_\eps,\varphi\rangle_{p',p}&\leq\eps^{\frac{1}{p'}}\||u_\eps|^{p-2}u_\eps\|_{L^{p'}(\otr)}\eps^{\frac{1}{p}}\|\varphi\|_{L^p(\otr)}\\
&\rightarrow 0\text{ as }\eps\downarrow0
\end{align*}
where $\langle\cdot,\cdot\rangle_{p',p}$ denotes $\langle\cdot,\cdot\rangle_{L^{p'}(\otr),L^p(\otr)}$. So
\begin{align}\label{convepsuepsuepsinlp'}
\eps|u_\eps|^{p-2}u_\eps\rightharpoonup0\text{ as }\eps\downarrow0\text{ in }L^{p'}(\otr)
\end{align}
is valid. Next, we want to study the convergence of $(\Delta_p(u_\eps))_{0<\varepsilon<1}$ in \\$L^{p'}(\Omega;L^{p'}(0,T;\sobolevq))$ as $\eps\downarrow 0$. Therefore we consider $(\nabla u_\eps)_{0<\varepsilon<1}$ first. The boundedness of $(\nabla u_\eps)_{0<\varepsilon<1}$ in $L^p(\otr)$ follows from \eqref{itoueps}. Indeed, $0<\eps<1$ we find
\begin{align}\label{nablauepsinlp}
\begin{split}
\int_0^T\erwb\|\nabla u_\eps\|_{p}^p\erwe\,ds=\,&\frac{1}{2}\erwb\|u_0\|_2^2\erwe-\frac{1}{2}\erwb\|u_\eps(T)\|_2^2\erwe-\eps\int_0^T\erwb\|u_\eps\|_p^p\erwe\,ds\\
&+\frac{1}{2}\int_0^T\erwb\|\Phi\|_\HS^2\erwe\,ds\leq\,\frac{C_1}{2}.
\end{split}
\end{align}
Hence, we can find a subsequence, still denoted by $(\nabla u_\eps)_{0<\varepsilon<1}$, and an element $v\in L^p(\otr)$ such that
\begin{align*}
\nabla u_\eps\rightharpoonup v\text{ as }\eps\downarrow0\text{ in }L^p(\otr).
\end{align*}
Recalling that $(u_\eps)_{0<\varepsilon<1}$ convergences weakly to $u$ in $L^2(\otr)$ as $\eps\downarrow0$, it is easy to see that $v=\nabla u$.

Thus we have
\begin{align}\label{convnablauepsinlp}
\nabla \ueps\rightarrow\nabla u\text{ as }\eps\downarrow0\text{ in }L^p(\otr).
\end{align}
In the next step we consider $(|\nabla u_\eps|^{p-2}\nabla u_\eps)_{0<\varepsilon<1}$ in $L^{p'}(\otr)^d$. Due to \eqref{nablauepsinlp} we have for $0<\eps<1$
\begin{align*}
\||\nabla u_\eps|^{p-2}\nabla u_\eps\|_{L^{p'}(\otr)}^{p'}&=\|\nabla u_\eps\|_{L^p(\otr)}^p\leq \frac{C_1}{2}.
\end{align*}
We conclude that there exists a subsequence, still denoted by\\ $(|\nabla u_\eps|^{p-2}\nabla u_\eps)_{0<\varepsilon<1}$, and $G\in L^{p'}(\otr)^d$ such that
\begin{align}\label{convuepsuepsinlp'}
|\nabla u_\eps|^{p-2}\nabla u_\eps \rightharpoonup G \text{ as }\eps\downarrow0\text{ in }L^{p'}(\otr).  
\end{align}
Now we consider $(\Delta_p(u_\eps))_{0<\varepsilon<1}$. By Hölder's inequality, for $0<\eps<1$ we get 
\begin{align*}
\|\Delta_p(u_\eps)\|_{-1,p'}&\leq\sup_{\varphi\in\sobolevp,\ \|\varphi\|_{1,p}\leq1}\int_{\re^d}|\nabla u_\eps|^{p-1}|\nabla\varphi|\dx\\
&\leq \sup_{\varphi\in\sobolevp,\ \|\varphi\|_{1,p}\leq1}\||\nabla u_\eps|^{p-1}\|_{p'}\|\nabla\varphi\|_p\\
&\leq \|\nabla u_\eps\|_p^{p-1}
\end{align*}
This yields together with \eqref{nablauepsinlp}
\begin{align*}
\|\Delta_p(\ueps)\|_{L^{p'}(\Omega\times(0,T);\sobolevq)}^{p'}&=\erwb\int_0^T\|\Delta_p(\ueps)\|_{-1,p'}^{p'}\, dt\erwe\\
&\leq\|\nabla u_\eps\|_{L^p(\otr)}^p\\
&\leq\frac{C_1}{2}.
\end{align*}
According to this we can find a subsequence, still denoted by $(\Delta_p(u_\eps))_{0<\varepsilon<1}$, and $U\in L^{p'}(\Omega;L^{p'}(0,T;\sobolevq))$ such that
\begin{align*}
\Delta_p(u_\eps)\rightharpoonup U\text{ as }\eps\downarrow0\text{ in }L^{p'}(\Omega;L^{p'}(0,T;\sobolevq)).
\end{align*}
From the Gauss-Green theorem and \eqref{convuepsuepsinlp'} now it follows that $U=\diver G$ and therefore
\begin{align*}
\Delta_p(\ueps)\rightharpoonup\diver G \text{ as }\eps\downarrow0\text{ in }L^{p'}(\Omega;L^{p'}(0,T;\sobolevq)).
\end{align*}
In the following, we want to use the achieved convergence results to pass to the limit in the regularized equation \eqref{approxequadd} for $\eps\downarrow0$. 
Since $u_\eps$ is a solution to \eqref{approxequadd}, we have
\begin{align*}
\frac{d}{dt}\left(u_\eps(t)-\int_0^t \Phi\,dW_s\right)=\Delta_p(u_\eps)-\eps|u_\eps|^{p-2}u_\eps \text{ in }\mathcal{D}'((0,T)\times\re^d)\text{ a.s. in }\Omega.
\end{align*}
We choose an arbitrary set $A\in\mathcal{F}$ and a test function $\varphi\in C_c^\infty([0,T]\times\re^d)$. Using $\varphi$ as a test function in the above equation and integrating over $A$ we arrive at  
\begin{align*}
&-\int_A\int_0^T\int_{\re^d}\left(u_\eps(t)-\int_0^t \Phi\,dW_s\right)\varphi_t\,dx\,dt\,dP\\
&=-\int_A\int_0^T\int_{\re^d}|\nabla u_\eps|^{p-2}\nabla u_\eps\cdot\nabla\varphi\,dx\,dt\,dP\\
&-\eps\int_A\int_0^T\int_{\re^d}|u_\eps|^{p-2}u_\eps\varphi\,dx\,dt\,dP.
\end{align*}
For $\eps\downarrow0$, from \eqref{convuepsinl2}, \eqref{convepsuepsuepsinlp'} and \eqref{convuepsuepsinlp'} it follows that
\begin{align*}
\int_A\int_0^T\int_{\re^d}\left(u(t)-\int_0^t \Phi\,dW_s\right)\varphi_t\,dx\,dt\,dP=\int_A\int_0^T\int_{\re^d} G\cdot\nabla \varphi\,dx\,dt\,dP.
\end{align*}
So we have
\begin{align*}
\frac{d}{dt}\left(u(t)-\int_0^t \Phi\,dW_s\right)=\diver G\text{ in }L^{p'}(0,T;W^{-1,p'}(\mathbb{R}^d))\text{ a.s. in }\Omega,
\end{align*}
and the above expression is equivalent to
\begin{align}\label{convapproxequadd}
u(t)-u_0-\int_0^t\diver G\,ds=\int_0^t \Phi\,dW_s \text { in } W^{-1,p'}(\mathbb{R}^d)\text{ for all } t\in [0,T] \text{ a.s. in }\Omega.
\end{align}
From \eqref{convuepsinl2} it follows that $\mathbb{E}\Vert u(t)\Vert^2_2< \infty$ a.e in $(0,T)$, hence from \cite[Theorem 4.2.5, p.91]{LiuRock} we get $u\in L^2(\Omega;\mathcal{C}([0,T];L^2(\mathbb{R}^d)))$.
Obviously we have completed the proof of Proposition \ref{200904_prop1} if we show $G=|\nabla u|^{p-2}\nabla u$. This can be done repeating the arguments of Lemma \ref{G=nablasu}. 
\end{proof}

\begin{flushleft}
\textbf{Acknowledgement:} This work has been supported by the German Research Foundation project ZI 1542/3-1.
\end{flushleft}

\bibliographystyle{unsrt}
\bibliography{Schmitz_Zimmermann_Bibliography.bib}
\end{document}